\documentclass[11pt]{amsart}

\usepackage[T1]{fontenc}
\usepackage[utf8]{inputenc}
\usepackage{lmodern}
\usepackage{amsmath,amssymb,amsthm,mathtools}
\usepackage{graphicx}
\usepackage{microtype}
\usepackage[colorlinks=true,linkcolor=blue,citecolor=blue,urlcolor=blue]{hyperref}
\usepackage[margin=1.15in]{geometry}
\usepackage{enumitem}

\graphicspath{{./}}

\newtheorem{theorem}{Theorem}[section]
\newtheorem{lemma}[theorem]{Lemma}
\newtheorem{proposition}[theorem]{Proposition}
\newtheorem{corollary}[theorem]{Corollary}

\newtheorem*{conjecture*}{Conjecture}
\theoremstyle{remark}
\newtheorem{remark}[theorem]{Remark}

\newcommand{\D}{\mathbb D}
\newcommand{\T}{\mathbb T}

\newcommand{\ii}{\mathrm i}
\newcommand{\dd}{\mathrm d}
\newcommand{\eps}{\varepsilon}
\newcommand{\zc}{z_\circ}
\newcommand{\Dc}{D_\circ}

\numberwithin{equation}{section}

\title[The Nitsche--Hopf conjecture for minimal graphs]{The Nitsche--Hopf Conjecture for Minimal Graphs}

\author{David Kalaj}
\address{Faculty of Natural Sciences and Mathematics, University of Montenegro, Podgorica, Montenegro}
\email{davidk@ucg.ac.me}

\author{Jian-Feng Zhu}
\address{Department of Mathematics, Shantou University, Shantou, Guangdong 515063, P. R. China}
\email{flandy@stu.edu.cn}

\date{May 2026}

\begin{document}

\begin{abstract}
We prove the Nitsche--Hopf conjecture for non-parametric minimal graphs over
disks.  If \(S\) is a minimal graph over a disk of radius \(R\), and if
\(\xi\) is the point above the center, then
\[
        W(\xi)^2 |K(\xi)|<\frac{\pi^2}{2R^2}.
\]
Here \(K\) is the Gaussian curvature and
\[
        W=\sqrt{1+|\nabla u|^2}=\frac1{n_3}
\]
is the reciprocal of the vertical component of the upward unit normal.  The constant is sharp, as shown by the horizontal tangent-plane extremal
sequence of Finn and Osserman.

The main difficulty is that the bicentric-quadrilateral comparison theorem
for Gaussian curvature controls \(|K|\), but it does not by itself control
the normalized quantity \(W^2|K|\): the slope factor \(W\) can be arbitrarily
large.  We show that the missing information is recovered inside the
Scherk-type comparison family from the zero equation for the horizontal
harmonic projection.  More precisely, in the fixed-arc normalization the
point corresponding to the center of the physical disk is a distinguished
zero \(z_\circ\) of the harmonic projection.  The equation
\(f(z_\circ)=0\), written in harmonic-measure coordinates, reduces the sharp
Hopf estimate to a scalar derivative inequality at the admissible zero of a
monotone function \(G_{A,B}\).  We prove this scalar inequality on the full
admissible parameter domain by a barrier argument and two explicit
Bernstein-polynomial positivity certificates.

Combined with the bicentric-quadrilateral comparison theorem of the first author and
Melentijevi\'c, the Scherk-family estimate gives the sharp
normalized Hopf estimate for arbitrary minimal graphs over disks.  As a
byproduct, we obtain the two-sided bound
\[
        \frac{\pi^2}{4}\le W^2|K|\leq \frac{\pi^2}{2}
\]
throughout the normalized Scherk-type comparison family, evaluated at the
distinguished point corresponding to the center. 
\end{abstract}

\maketitle

\section{Introduction and main results}

A classical problem in the theory of non-parametric minimal surfaces asks for
the sharp curvature bound at the center of a minimal graph over a disk.  Let
\[
        \Sigma=\{(x,y,u(x,y)):(x,y)\in D_R\}\subset\mathbb R^3
\]
be a solution of the minimal surface equation over the disk of radius \(R\).
Its Gaussian curvature is
\[
        K=\frac{u_{xx}u_{yy}-u_{xy}^2}
        {(1+u_x^2+u_y^2)^2},
\]
and the inclination factor is
\[
        W=\sqrt{1+|\nabla u|^2}.
\]
Equivalently, \(W=1/n_3\), where \(n_3\) is the vertical component of the
upward unit normal.  For general background on minimal surfaces and harmonic
mappings, see \cite{AlarconForstnericLopez2021,Duren2004,Nitsche1975}; related
sharp-curvature estimates are discussed in \cite{Hall1998}.

Heinz proved that the curvature at the center satisfies
\[
        |K(0)|\leq \frac{C}{R^2}
\]
for an absolute constant \(C\); see \cite{Heinz1952}.  Hopf later refined this result by showing that the
correct scale-invariant quantity is not merely \(R^2|K(0)|\), but rather
\[
        R^2W(0)^2|K(0)|.
\]
More precisely, Hopf proved \cite{Hopf1953} that
\[
        W(0)^2|K(0)|\leq \frac{C_1}{R^2}
\]
for an absolute constant \(C_1\).  The remaining problem is to determine the sharp value of this universal
constant.

Finn and Osserman proved that the sharp constant is \(\pi^2/2\) in the
horizontal tangent-plane case, that is, when \(W(0)=1\)
\cite{FinnOsserman1964}.  Nitsche observed that the same constant should
remain sharp for arbitrary inclination and verified this assertion in several
symmetric classes \cite{Nitsche1973}.  This led to the conjecture that
\[
        R^2W(0)^2|K(0)|<\frac{\pi^2}{2}
\]
for every minimal graph over a disk, with \(\pi^2/2\) being the best possible
constant.  We shall refer to this sharp-constant conjecture as the
{\it Nitsche--Hopf conjecture}.

The purpose of this paper is to prove the Nitsche--Hopf conjecture.

\begin{theorem}[Normalized Hopf curvature estimate]\label{thm:main}
Let \(S\) be a non-parametric minimal graph over the disk of radius \(R\),
and let \(\xi\) be the point of \(S\) above the center of the disk.  Then
\[
        W_S(\xi)^2 |K_S(\xi)|<\frac{\pi^2}{2R^2}.
\]
The constant \(\pi^2/2\) is sharp.
\end{theorem}

The strict inequality in Theorem~\ref{thm:main} is inherited from the strict
form of the comparison theorem used below.  The sharpness follows from the
horizontal tangent-plane extremal sequence of Finn and Osserman.

The main difficulty is that the normalized problem is genuinely stronger
than the corresponding problem for \(|K|\).  A pointwise estimate
\[
        |K(\xi)|\leq \frac{C}{R^2}
\]
does not give useful information on \(W(\xi)^2|K(\xi)|\), since the slope
factor \(W(\xi)\) can be arbitrarily large.  In conformal coordinates this
obstruction is particularly transparent.  If \(g\) denotes the Gauss-map
parameter of an upward oriented minimal graph, then
\[
        W=\frac{1+|g|^2}{1-|g|^2}.
\]
Thus \(W\) becomes large as the Gauss map approaches the unit circle.  The
Hopf problem therefore requires information which is not contained in a
comparison estimate for \(|K|\) alone.

The starting point of the proof is the bicentric-quadrilateral comparison
theorem of Kalaj and Melentijevi\'c \cite{KalajMelentijevic2026}.  That
theorem reduces sharp curvature questions for arbitrary disk graphs to a
two-parameter family of Scherk-type minimal graphs over bicentric
quadrilaterals.  The corresponding admissible parameter set
\[
        \mathcal S\subset\{(p,q):0<p<q<\pi\}
\]
was determined in \cite{KalajMelentijevic2026}; its shape is shown in
Figure~\ref{fig:domain}.  Each point \((p,q)\in\mathcal S\) determines a
Scherk-type comparison graph.

For the unnormalized curvature, the comparison theorem is sufficient: after
a normalization and a horizontal rotation, the curvature of a given minimal
graph at the point above the center is strictly dominated by the curvature
of the corresponding Scherk-type comparison graph.  For the normalized
quantity one additional fact is decisive.  The comparison graph has the same
unit normal at the distinguished point as the original graph.  Hence, the two
graphs have the same value of \(W\) at that point.  Therefore, the strict
curvature comparison transfers immediately to \(W^2|K|\), provided one proves
the sharp normalized estimate inside the Scherk-type comparison family.

This Scherk-family estimate is the new ingredient of the paper.

\begin{theorem}[Scherk-type normalized Hopf estimate]\label{thm:scherk-main}
For every admissible two-parameter Scherk-type comparison graph,
\[
        W^2|K|\leq \frac{\pi^2}{2}
\]
at the distinguished point corresponding to the center in the comparison
normalization.
\end{theorem}

The central observation is that the Scherk comparison graph is not evaluated
at an arbitrary point of its conformal parameter disk.  We use the fixed-arc
normalization of the Scherk family.  In this normalization the horizontal
projection \(f\) is explicit: it is the harmonic extension of a four-step
boundary function whose values are the vertices of the bicentric quadrilateral
determined by the admissible pair \((p,q)\in\mathcal S\).  Thus
\[
        f(z)=\sum_{j=1}^4 \Omega_j(z)P_j,
\]
where \(P_1,\ldots,P_4\) are the four vertices of the associated bicentric
quadrilateral and \(\Omega_j(z)\) are the harmonic measures of the four fixed
boundary arcs.

The point relevant for the Hopf estimate is the point of the surface lying
above the center of the physical domain.  In the fixed-arc parameter disk its
preimage is a distinguished point \(z_\circ\), characterized by
\[
        f(z_\circ)=0.
\]
This zero equation is the additional constraint which compensates for the
possible largeness of the inclination factor \(W\).

The analytic form of the problem comes from the Enneper--Weierstrass data.
If
\[
        X(z)=\operatorname{Re}\int^z
        \bigl(1-g^2,\ii(1+g^2),2g\bigr)\varphi\,\dd z,
\]
then, for an upward oriented minimal graph,
\[
        K=-\frac{4|g'|^2}{|\varphi|^2(1+|g|^2)^4},
        \qquad
        W=\frac{1+|g|^2}{1-|g|^2}.
\]
For the fixed-arc Scherk family, the explicit Weierstrass data computed in
\cite{KalajMelentijevic2026} give the following
\[
        -K
        =
        \frac{\pi^2}{4}(1+\mu^2)
        \frac{|1-z^2|^2|z^2-e^{2\ii\alpha}|^2}{D(z)^4},
\]
where
\[
        \mu^2=AB,\qquad A=\sin p,\qquad B=\sin(q-p),
\]
and \(D(z)>0\) is an explicit factor depending on the Gauss-map parameter
and on \(z\).  Since the Gauss map is a disk automorphism, one also has, at
\(z_\circ=re^{\ii t_\circ}\),
\[
        W(z_\circ)=\frac{D_\circ}{\mu(1-r^2)},
        \qquad
        D_\circ=D(z_\circ).
\]
Multiplying the curvature formula by this expression for \(W^2\) gives
\[
        W^2|K|
        =
        \frac{\pi^2}{4}\frac{1+\mu^2}{\mu^2}
        \frac{|1-z_\circ^2|^2|z_\circ^2-e^{2\ii\alpha}|^2}
        {(1-|z_\circ|^2)^2D_\circ^2}.
\]
For arbitrary \(z\in\mathbb D\), this identity alone does not imply Hopf's
sharp bound.  The decisive input is precisely the zero condition
\(f(z_\circ)=0\).

We exploit this condition by writing it in harmonic-measure coordinates.
Introduce
\[
        U=\Omega_1(z_\circ)+\Omega_3(z_\circ),
        \qquad
        V=\Omega_1(z_\circ)-\Omega_3(z_\circ),
        \qquad
        T=\Omega_4(z_\circ)-\Omega_2(z_\circ).
\]
The zero equation is a linear relation among the four harmonic measures.  In
the variables \(U,V,T\), it makes \(V\) and \(T\) affine functions of the
single variable \(U\).  The fact that these four numbers are harmonic
measures of the prescribed fixed arcs gives a cross-ratio condition, and
therefore one scalar equation for \(U\).

After eliminating the remaining phase term in \(D_\circ\), the sharp Hopf
estimate becomes a one-variable derivative inequality at an admissible zero
of an explicit monotone function.  We now state this scalar form.

Let
\[
        0<A,B\leq 1,
        \qquad
        \kappa=\sqrt{1-A^2},
        \qquad
        \varepsilon=\sqrt{1-B^2}.
\]
Set
\[
        P=P(A,B):=\frac{1+AB}{B(A+B)}.
\]
Define
\[
        M(U):=\kappa(P-U),
\]
\[
        N(U):=\varepsilon\left(U+\frac{\kappa^2}{A(A+B)}\right),
\]
and
\[
        G_{A,B}(U)
        :=
        B\cos(\pi M(U))
        -A\cos(\pi N(U))
        -(A+B)\cos(\pi U).
\]
The admissible interval is
\[
        L(A,B)\leq U\leq R(A,B),
\]
where
\[
        L(A,B)=
        \frac{\kappa}{1+\kappa}\frac{1+AB}{B(A+B)}
\]
and
\[
        R(A,B)=
        \frac{1-\varepsilon\kappa^2/[A(A+B)]}{1+\varepsilon}.
\]
This interval is nonempty precisely when
\[
        B\geq B_0(A),
\]
with
\[
        B_0(A)=
        \frac{-A(1-\kappa)+
        \sqrt{A^2(1-\kappa)^2+8\kappa(1+\kappa)}}
        {2(1+\kappa)}.
\]

\begin{theorem}[Scalar derivative inequality]\label{thm:scalar}
Assume \(0<A,B\leq 1\) and \(B\geq B_0(A)\).  Let \(U\) be an admissible
zero of \(G_{A,B}\), that is,
\[
        G_{A,B}(U)=0,
        \qquad
        L(A,B)\leq U\leq R(A,B).
\]
Then
\[
        \frac{1}{\pi}G'_{A,B}(U)
        \geq
        \sqrt{2(1+AB)}.
\]
Equivalently,
\[
\begin{aligned}
 &(A+B)\sin(\pi U)
        +B\kappa\sin(\pi M(U))
        +A\varepsilon\sin(\pi N(U))       \\
 &\hspace{4cm}\geq \sqrt{2(1+AB)}.
\end{aligned}
\]
\end{theorem}

The proof of Theorem~\ref{thm:scalar} is elementary but rigid.  The
admissibility inequalities imply
\[
        0\leq M(U),N(U)\leq \frac12.
\]
Thus the sine terms can be bounded below by linear functions, reducing the
problem to two linear estimates.  Each of these estimates is obtained from a
barrier point for the monotone function \(G_{A,B}\).  The only global
positivity input required in the argument consists of two explicit
Bernstein-polynomial certificates on the square \([0,1]^2\), recorded in
Appendix~\ref{app:bernstein}.

The logical structure of the proof is as follows.  First, the fixed-arc
Scherk normalization gives the exact identity for \(W^2|K|\).  Second, the
zero equation \(f(z_\circ)=0\) is rewritten in harmonic-measure coordinates,
reducing the geometry to the scalar equation
\[
        G_{A,B}(U)=0.
\]
Third, harmonic-measure identities eliminate the phase appearing in
\(D_\circ\), and the Hopf estimate becomes precisely the derivative
inequality in Theorem~\ref{thm:scalar}.  Fourth, the scalar theorem proves
Theorem~\ref{thm:scherk-main}.  Finally, the strict comparison theorem of Kalaj and Melentijevi\'c transfers
the Scherk-family estimate to arbitrary minimal graphs over disks and gives
Theorem~\ref{thm:main}.

The paper is organized as follows.  Section~\ref{sec:scherk} recalls the
Scherk comparison normalization and derives the exact normalized curvature
identity.  Section~\ref{sec:harmonic} rewrites the zero equation
\(f(z_\circ)=0\) in harmonic-measure coordinates and eliminates the phase
term.  Section~\ref{sec:scalar-proof} proves the scalar derivative
inequality.  Section~\ref{sec:completion} completes the proof of the
normalized Hopf estimate.  Appendix~\ref{app:odd} records an auxiliary sharp
coefficient estimate for odd harmonic disk maps.  Appendix~\ref{app:lower}
gives the complementary lower bound in the Scherk family.  Appendix~\ref{app:bernstein} contains the Bernstein-polynomial positivity
certificates.  Appendix~\ref{app:log} records a log-subharmonicity
observation for \(W^2|K|\).

\section{The Scherk comparison family and the exact identity}\label{sec:scherk}

\subsection{The Scherk comparison theorem}

We use the following comparison theorem \cite[Theorem 1.4]{KalajMelentijevic2026}.  For $w\in\D$, put
\[
        \lambda_w=\frac{\ii(1-w^4)}{|1-w^4|},
        \qquad
        \mu_w(z)=\left(\frac{w+\lambda_w z}{1+\overline w\lambda_w z}\right)^2.
\]

\begin{theorem}[Scherk comparison principle]\label{thm:comparison}
For every $w\in\D$, there are four distinct points
\[
        \zeta_0,\zeta_1,\zeta_2,\zeta_3\in\T
\]
and a sense-preserving harmonic diffeomorphism
\[
        f^\diamond:\D\to Q(\zeta_0,\zeta_1,\zeta_2,\zeta_3)
\]
onto the corresponding quadrilateral such that
\[
        f^\diamond(0)=0,
        \qquad
        (f^\diamond)_z(0)>0,
\]
and
\[
        \overline{(f^\diamond)_{\bar z}(z)}=\mu_w(z)(f^\diamond)_z(z).
\]
Equivalently, if $f^\diamond=h^\diamond+\overline{g^\diamond}$, then
\[
        (g^\diamond)'(z)=\mu_w(z)(h^\diamond)'(z).
\]
The harmonic map $f^\diamond$ gives a Scherk-type minimal graph $S^\diamond$ over $Q(\zeta_0,\zeta_1,\zeta_2,\zeta_3)$ containing the point $\xi=(0,0,0)$. Its upward normal at $\xi$ is
\[
        n^\diamond_\xi
        =-\frac{1}{1+|w|^2}\bigl(2\operatorname{Im}w,2\operatorname{Re}w,-1+|w|^2\bigr),
\]
and $U^\diamond_{uv}(0,0)=0$. If $S$ is any other non-parametric minimal graph over the unit disk satisfying
\[
        (0,0,0)\in S,
        \qquad
        n_\xi=n^\diamond_\xi,
        \qquad
        U_{uv}(0,0)=0,
\]
then
\[
        |K_S(\xi)|< |K_{S^\diamond}(\xi)|.
\]
Moreover,
\[
        K_{S^\diamond}(\xi)
        =-\frac{4(1-|w|^2)^2}{(1+|w|^2)^4 |(f^\diamond)_z(0)|^2}.
\]
\end{theorem}

Because the compared surfaces have the same normal, the theorem immediately transfers to the normalized quantity.

\begin{corollary}[Normalized comparison]\label{cor:normalized-comparison}
Under the hypotheses of Theorem~$\ref{thm:comparison}$,
\[
        W_S(\xi)^2 |K_S(\xi)|
        <
        W_{S^\diamond}(\xi)^2 |K_{S^\diamond}(\xi)|.
\]
Furthermore,
\[
        W_{S^\diamond}(\xi)^2 |K_{S^\diamond}(\xi)|
        =\frac{4}{(1+|w|^2)^2 |(f^\diamond)_z(0)|^2}.
\]
\end{corollary}

\begin{proof}
The two surfaces have the same unit normal at $\xi$. Since $W=1/n_3$, they also have the same value of $W$ at $\xi$. Multiplying the curvature comparison in Theorem~\ref{thm:comparison} by this common factor $W^2$ gives the first assertion.

From the formula for $n^\diamond_\xi$,
\[
        (n^\diamond_\xi)_3=\frac{1-|w|^2}{1+|w|^2},
        \qquad
        W_{S^\diamond}(\xi)=\frac{1+|w|^2}{1-|w|^2}.
\]
Multiplication by the curvature formula in Theorem~\ref{thm:comparison} gives the stated identity.
\end{proof}

\begin{remark}[Two normalizations]\label{rem:normalizations}
The comparison theorem is stated in the normalization $f^\diamond(0)=0$. The fixed-arc normalization used below is obtained by precomposition with a disk automorphism. In that normalization the Beltrami coefficient becomes $z^2$, and the preimage of the point above the origin is generally not $0$. We denote this preimage by $\zc$ and write
\[
        f(\zc)=0.
\]
The quantity $W^2|K|$ is intrinsic to the surface point, and therefore is unchanged by this reparametrization.
\end{remark}

\subsection{Parameters}

The fixed-arc Scherk family is described by two admissible angular parameters $p,q$. We use
\[
        A=\sin p,
        \qquad
        B=\sin(q-p),
        \qquad
        c_p=\cos p,
        \qquad
        d_q=\cos(q-p),
\]
and
\[
        \mu^2=AB,
        \qquad
        0<A\leq 1,
        \qquad
        0<B\leq 1.
\]
The arc parameter $\alpha$ is determined by
\[
        \tan^2\frac{\alpha}{2}=\frac{A}{B},
\]
so that
\[
        \sin\alpha=\frac{2\sqrt{AB}}{A+B}=\frac{2\mu}{A+B}.
\]
We write
\[
        \zc=re^{\ii t_\circ},
        \qquad
        0\leq r<1,
\]
for the distinguished zero of the harmonic projection.

The Gauss-map parameter is an automorphism of the disk. We use the standard convention
\[
        g(z)=e^{\ii\vartheta}\frac{z-a}{1-\overline a z},
        \qquad
        a=|a|e^{\ii\delta},
\]
where
\[
        |a|=\sqrt{\frac{1-\mu}{1+\mu}}.
\]
Consequently,
\[
        \frac{1+|a|^2}{1-|a|^2}=\frac1\mu,
        \qquad
        \frac{4|a|}{1+|a|^2}=2\sqrt{1-\mu^2}.
\]

\subsection{Exact normalized curvature identity}

For a minimal graph written in conformal Enneper--Weierstrass form, one has
\[
        W=\frac{1+|g|^2}{1-|g|^2}.
\]
At $z=\zc$ this gives
\[
        W=\frac{1+|g(\zc)|^2}{1-|g(\zc)|^2}.
\]
Using the automorphism formula for $g$, we have
\[
        1-|g(\zc)|^2
        =\frac{(1-|a|^2)(1-r^2)}{|1-\overline a\zc|^2},
\]
and
\[
        1+|g(\zc)|^2
        =\frac{|1-\overline a\zc|^2+|\zc-a|^2}{|1-\overline a\zc|^2}.
\]
Hence,
\[
        W=\frac{|1-\overline a\zc|^2+|\zc-a|^2}{(1-|a|^2)(1-r^2)}.
\]
The numerator can be rewritten as follows
\[
\begin{aligned}
 |1-\overline a\zc|^2+|\zc-a|^2
 &= (1+r^2)(1+|a|^2)-4\operatorname{Re}(a\overline\zc)  \\
 &= (1+|a|^2)\left(1+r^2-\frac{4\operatorname{Re}(a\overline\zc)}{1+|a|^2}\right).
\end{aligned}
\]
Thus
\[
        W=\frac{\Dc}{\mu(1-r^2)},
\]
where
\[
        \Dc=1+r^2-\frac{4\operatorname{Re}(a\overline\zc)}{1+|a|^2}.
\]
Equivalently,
\[
        \Dc=1+r^2-2\sqrt{1-\mu^2}\, r\cos(t_\circ-\delta).
\]

\begin{theorem}[Exact normalized curvature identity]\label{thm:exact}
Let $\zc=re^{\ii t_\circ}$ be the distinguished zero of the harmonic projection.  Then
\[
        W^2|K|
        =\frac{\pi^2}{4}\frac{1+
        \mu^2}{\mu^2}
        \frac{|1-\zc^2|^2|\zc^2-e^{2\ii\alpha}|^2}{(1-r^2)^2\Dc^2}.
\]
\end{theorem}

\begin{proof}
The explicit curvature formula for the fixed-arc two-parameter Scherk family is
\[
        -K
        =\frac{\pi^2}{4}(1+\mu^2)
        \frac{|1-\zc^2|^2|\zc^2-e^{2\ii\alpha}|^2}{\Dc^4};
\]
see \cite[Eq. (4.2)]{KalajMelentijevic2026}.  Multiplying this identity by
\[
        W^2=\frac{\Dc^2}{\mu^2(1-r^2)^2}
\]
gives the formula.
\end{proof}

\begin{corollary}[Zero-control formulation]\label{cor:zero-control}
The estimate
\[
        W^2|K|\leq \frac{\pi^2}{2}
\]
is equivalent to
\[
        |1-\zc^2|\,|\zc^2-e^{2\ii\alpha}|
        \leq
        \sqrt{\frac{2\mu^2}{1+\mu^2}}
        (1-r^2)\Dc.
\]
\end{corollary}

\begin{remark}
The inequality in Corollary~\ref{cor:zero-control} is not intended for arbitrary $z\in\D$.  The decisive extra information is that $z=\zc$ is the zero of the harmonic projection, $f(\zc)=0$.
\end{remark}

\section{Harmonic-measure coordinates and phase elimination}\label{sec:harmonic}

\subsection{The zero equation}

Let
\[
        I_1=(0,\alpha),\quad
        I_2=(\alpha,\pi),\quad
        I_3=(\pi,\pi+\alpha),\quad
        I_4=(\pi+\alpha,2\pi),
\]
and put
\[
        \Omega_j(z)=\omega(z,I_j,\D),
        \qquad j=1,2,3,4.
\]
In fixed-arc normalization, the boundary values of the harmonic projection are the four constant values
\[
        P_1=e^{\ii(q-p+2\beta)},
        \qquad
        P_2=e^{-\ii p},
        \qquad
        P_3=e^{\ii(p-q+2\beta)},
        \qquad
        P_4=e^{\ii p}.
\]
Since $f$ is the Poisson extension of this step function, we see that
\[
        f(z)=\sum_{j=1}^4 \Omega_j(z)P_j.
\]
Therefore,
\[
        f(\zc)=0
        \quad\Longleftrightarrow\quad
        \sum_{j=1}^4 \Omega_j(\zc)P_j=0.
\]
Introduce
\[
        U=\Omega_1+
        \Omega_3,
        \qquad
        V=\Omega_1-
        \Omega_3,
        \qquad
        T=\Omega_4-
        \Omega_2.
\]
Then
\[
        \Omega_1=\frac{U+V}{2},\quad
        \Omega_3=\frac{U-V}{2},\quad
        \Omega_4=\frac{1-U+T}{2},\quad
        \Omega_2=\frac{1-U-T}{2}.
\]
The positivity of the harmonic measure gives
\[
        |V|\leq U,
        \qquad
        |T|\leq 1-U.
\]
The zero equation becomes
\[
        e^{2\ii\beta}\bigl(U\cos(q-p)+\ii V\sin(q-p)\bigr)
        +(1-U)\cos p+
        \ii T\sin p=0.
\]
Equivalently,
\[
        V=c_p\left(\frac{1+\mu^2}{B(A+B)}-U\right),
\]
and
\[
        T=-d_q\left(U+\frac{c_p^2}{A(A+B)}\right).
\]
Thus, the zero equation leaves only one free harmonic-measure parameter, namely $U$.

\subsection{The numerator and the master inequality}

The map $z\mapsto z^2$ sends $I_1\cup I_3$ onto the single arc $(0,2\alpha)$. Hence,
\[
        U=\omega(\zc^2,(0,2\alpha),\D).
\]
For $w=\zc^2$, the harmonic-measure formula for one circular arc gives
\[
        \sin(\pi U)
        =\frac{(1-|w|^2)\sin\alpha}{|1-w|\,|e^{2\ii\alpha}-w|}.
\]
Since $|w|=r^2$,
\[
        \frac{|1-\zc^2|\,|\zc^2-e^{2\ii\alpha}|}{1-r^2}
        =\frac{(1+r^2)\sin\alpha}{\sin(\pi U)}.
\]
Substitution into Theorem~\ref{thm:exact} yields
\[
        W^2|K|
        =\frac{\pi^2}{4}\frac{1+\mu^2}{\mu^2}
        \left(\frac{(1+r^2)\sin\alpha}{\Dc\sin(\pi U)}\right)^2.
\]
Using
\[
        \frac{\Dc}{1+r^2}
        =1-\sqrt{1-\mu^2}\frac{2r}{1+r^2}\cos(t_\circ-\delta),
\]
we see that $W^2|K|\leq \pi^2/2$ is equivalent to
\[
        \sin(\pi U)
        \left(1-\sqrt{1-\mu^2}\frac{2r}{1+r^2}\cos(t_\circ-\delta)\right)
        \geq
        \frac{\sqrt{2(1+
        \mu^2)}}{A+B}.
\]
This is the master inequality.

\subsection{The scalar equation for \texorpdfstring{$U$}{U}}
\begin{lemma}[Harmonic-measure cross-ratio condition]\label{lem:hm-crossratio}
Let \(z\in\mathbb D\), and let
\[
        \Omega_j=\omega(z,I_j,\mathbb D),
        \qquad j=1,2,3,4,
\]
where
\[
        I_1=(0,\alpha),\quad
        I_2=(\alpha,\pi),\quad
        I_3=(\pi,\pi+\alpha),\quad
        I_4=(\pi+\alpha,2\pi).
\]
Then
\[
\frac{
\sin(\pi\Omega_1)\sin(\pi\Omega_3)}
{\sin(\pi\Omega_2)\sin(\pi\Omega_4)}
        =\tan^2\frac{\alpha}{2}.
\]
Equivalently, in the variables
\[
        U=\Omega_1+\Omega_3,\qquad
        V=\Omega_1-\Omega_3,\qquad
        T=\Omega_4-\Omega_2,
\]
one has
\[
\frac{
\sin\frac{\pi}{2}(U+V)
\sin\frac{\pi}{2}(U-V)}
{
\sin\frac{\pi}{2}(1-U-T)
\sin\frac{\pi}{2}(1-U+T)}
        =\tan^2\frac{\alpha}{2}.
\]
\end{lemma}

\begin{proof}
Let \(\phi\) be a disk automorphism with \(\phi(z)=0\).  Harmonic measure is
conformally invariant, so
\[
        \Omega_j=\omega(0,\phi(I_j),\mathbb D).
\]
At the origin, harmonic measure is normalized arclength.  Hence, the four arcs
\(\phi(I_j)\) have angular lengths \(2\pi\Omega_j\).

Let the endpoints of the four image arcs be
\[
        \eta_0,\eta_1,\eta_2,\eta_3\in\mathbb T
\]
in cyclic order.  Since the angular length of the arc from \(\eta_0\) to
\(\eta_1\) is \(2\pi\Omega_1\), and similarly for the other arcs, we have
\[
        |\eta_0-\eta_1|=2\sin(\pi\Omega_1),\qquad
        |\eta_1-\eta_2|=2\sin(\pi\Omega_2),
\]
\[
        |\eta_2-\eta_3|=2\sin(\pi\Omega_3),\qquad
        |\eta_3-\eta_0|=2\sin(\pi\Omega_4).
\]
The absolute cross-ratio
\[
        \left|
        \frac{(\eta_0-\eta_1)(\eta_2-\eta_3)}
        {(\eta_1-\eta_2)(\eta_3-\eta_0)}
        \right|
\]
is invariant under M\"obius transformations.  Therefore, it is equal to the
corresponding cross-ratio of the original endpoints
\[
        1,\quad e^{\ii\alpha},\quad -1,\quad -e^{\ii\alpha}.
\]
Thus
\[
\begin{aligned}
\frac{\sin(\pi\Omega_1)\sin(\pi\Omega_3)}
     {\sin(\pi\Omega_2)\sin(\pi\Omega_4)}
&=
\left|
\frac{(1-e^{\ii\alpha})(-1+e^{\ii\alpha})}
{(e^{\ii\alpha}+1)(-e^{\ii\alpha}-1)}
\right|    \\
&=
\frac{4\sin^2(\alpha/2)}{4\cos^2(\alpha/2)}
=
\tan^2\frac{\alpha}{2}.
\end{aligned}
\]
Finally, substituting
\[
        \Omega_1=\frac{U+V}{2},\quad
        \Omega_3=\frac{U-V}{2},\quad
        \Omega_4=\frac{1-U+T}{2},\quad
        \Omega_2=\frac{1-U-T}{2}
\]
gives the stated form.
\end{proof}
By Lemma~\ref{lem:hm-crossratio}, the variables \(U,V,T\) satisfy
\[
\frac{
\sin\frac{\pi}{2}(U+V)
\sin\frac{\pi}{2}(U-V)}{
\sin\frac{\pi}{2}(1-U-T)
\sin\frac{\pi}{2}(1-U+T)}
        =\tan^2\frac{\alpha}{2}.
\]
Equivalently,
\[
        \cos(\pi V)-\cos(\pi U)
        =\tan^2\frac{\alpha}{2}\bigl(\cos(\pi T)+\cos(\pi U)\bigr).
\]
Since $\tan^2(\alpha/2)=A/B$, this equation becomes
\[
        B\cos(\pi V)-A\cos(\pi T)-(A+B)\cos(\pi U)=0.
\]
Define
\[
        \kappa=|c_p|=\sqrt{1-A^2},
        \qquad
        \eps=|d_q|=\sqrt{1-B^2},
\]
\[
        M(U)=|V(U)|=\kappa\left(\frac{1+AB}{B(A+B)}-U\right),
\]
and
\[
        N(U)=|T(U)|=\eps\left(U+\frac{\kappa^2}{A(A+B)}\right).
\]
Then the scalar equation is as follows
\[
        G_{A,B}(U)=0,
\]
where
\[
        G_{A,B}(U)
        =B\cos(\pi M(U))-A\cos(\pi N(U))-(A+B)\cos(\pi U).
\]
The admissibility inequalities are
\[
        0\leq M(U)\leq U,
        \qquad
        0\leq N(U)\leq 1-U.
\]
They are equivalent to
\[
        L(A,B)\leq U\leq R(A,B),
\]
where
\[
        L(A,B)=\frac{\kappa}{1+
        \kappa}\frac{1+AB}{B(A+B)},
\]
and
\[
        R(A,B)=\frac{1-\eps\kappa^2/[A(A+B)]}{1+
        \eps}.
\]

\begin{lemma}[Admissible domain]\label{lem:domain}
For $0<A,B\leq 1$, the interval $[L(A,B),R(A,B)]$ is nonempty if and only if
\[
        B\geq B_0(A),
\]
where
\[
        B_0(A)=
        \frac{-A(1-\kappa)+\sqrt{A^2(1-\kappa)^2+8\kappa(1+\kappa)}}{2(1+\kappa)}.
\]
Moreover, if $B\geq B_0(A)$, then
\[
        L(A,B)\leq \frac12\leq R(A,B),
\]
and, for every admissible zero,
\[
        0\leq M(U),N(U)\leq \frac12.
\]
\end{lemma}

\begin{proof}
First,
\[
        1-R(A,B)
        =\frac{\eps}{1+
        \eps}\left(1+\frac{\kappa^2}{A(A+B)}\right)
        =\frac{\eps}{1+
        \eps}\frac{1+AB}{A(A+B)}
        =L(B,A).
\]
Hence
\[
        L(A,B)\leq R(A,B)
        \quad\Longleftrightarrow\quad
        L(A,B)+L(B,A)\leq 1.
\]
 Inequality $L(A,B)\leq 1/2$ is equivalent to
\[
        2\kappa(1+AB)\leq (1+
        \kappa)B(A+B),
\]
or
\[
        (1+
        \kappa)B^2+A(1-
        \kappa)B-2\kappa\geq 0.
\]
The left-hand side is increasing in $B$ on $[0,1]$, and its positive root is $B_0(A)$.
A symmetric squaring computation gives
\[
        L(A,B)\leq \frac12
        \quad\Longleftrightarrow\quad
        L(B,A)\leq \frac12.
\]
Indeed, both are equivalent to the nonpositivity of
\[
        (1-A^2)(2+AB-B^2)^2-B^2(A+B)^2,
\]
which is invariant under interchanging $A$ and $B$ after using $A^2+
        \kappa^2=B^2+\eps^2=1$.
Thus
\[
        L(A,B)\leq R(A,B)
        \quad\Longleftrightarrow\quad
        L(A,B)\leq \frac12
        \quad\Longleftrightarrow\quad
        B\geq B_0(A).
\]
If $B\geq B_0(A)$, then $L(A,B)\leq 1/2$ and $L(B,A)\leq 1/2$, so $R(A,B)=1-L(B,A)\geq 1/2$.

We shall also use that the right endpoint lies below $P=(1+AB)/(B(A+B))$. Indeed,
\[
        P-R(A,B)=
        \frac{\eps\,(A\eps+A+B)}{AB(A+B)(1+\eps)}\geq 0.
\]
Thus, $U\leq R(A,B)$ implies $P-U\geq0$, so the formula $M(U)=\kappa(P-U)$ is nonnegative on the admissible interval.

Finally, $M$ is decreasing and $N$ is increasing. Hence, on the admissible interval,
\[
        M(U)\leq M(L)=L\leq \frac12,
        \qquad
        N(U)\leq N(R)=1-R\leq \frac12.
\]
The lower bounds are precisely the admissibility conditions. This proves the lemma.
\end{proof}

\begin{lemma}[Monotonicity]\label{lem:monotone}
The function $G_{A,B}$ is strictly increasing on every non-degenerate admissible interval.
\end{lemma}

\begin{proof}
Differentiating and using $M'(U)=-\kappa$, $N'(U)=\eps$, one obtains
\[
\begin{aligned}
        \frac1\pi G'_{A,B}(U)
        &=(A+B)\sin(\pi U)
        +B\kappa\sin(\pi M(U))
        +A\eps\sin(\pi N(U)).
\end{aligned}
\]
On the admissible interval all three sine factors are nonnegative. On every nontrivial subinterval at least one term is positive, so $G_{A,B}$ is strictly increasing.
\end{proof}

\subsection{Exact phase elimination}

Put
\[
        h=\frac{\alpha}{2}.
\]
For the arc centered at $\phi$ with half-length $s$, the harmonic-measure formula is
\[
        \cot(\pi\omega)
        =\frac{(1+r^2)\cos s-2r\cos(t_\circ-\phi)}{(1-r^2)\sin s}.
\]
Applying this formula to $I_1$ and $I_3$, whose centers are $h$ and $h+\pi$, gives
\[
        \frac{2r}{1+r^2}\cos(t_\circ-h)
        =\cos h\,\frac{\sin(\pi V)}{\sin(\pi U)}.
\]
Applying the same formula to $I_2$ and $I_4$, whose centers are $h+\pi/2$ and $h+3\pi/2$, gives
\[
        \frac{2r}{1+r^2}\cos\left(t_\circ-h-\frac\pi2\right)
        =-\sin h\,\frac{\sin(\pi T)}{\sin(\pi U)}.
\]
Consequently,
\[
\frac{2r}{1+r^2}\cos(t_\circ-\delta)
=
\frac{
\cos h\sin(\pi V)\cos(\delta-h)-
\sin h\sin(\pi T)\sin(\delta-h)}{\sin(\pi U)}.
\]

\begin{lemma}[Phase of $a$]\label{lem:phase-a}
Assume $AB<1$. With the notation above,
\[
        \cos(\delta-h)
        =-
        \frac{c_p\sqrt B}{\sqrt{(1-AB)(A+B)}},
\]
and
\[
        \sin(\delta-h)
        =-
        \frac{d_q\sqrt A}{\sqrt{(1-AB)(A+B)}}.
\]
\end{lemma}

\begin{proof}
The explicit formula for the parameter $a$ in the two-parameter Scherk family gives
\[
        a=\frac{d_q-c_p-\ii\sin q}{1-
        \cos q+2\sqrt{AB}+\ii(B-A)}.
\]
Using
\[
        \sin q=Ad_q+Bc_p,
        \qquad
        \cos q=c_pd_q-AB,
\]
one simplifies this expression to
\[
        a=\frac{Ad_q-Bc_p-
        \ii\sqrt{AB}(c_p+d_q)}{(1+
        \sqrt{AB})(A+B)}.
\]
Also,
\[
        |a|=\sqrt{\frac{1-
        \sqrt{AB}}{1+
        \sqrt{AB}}},
        \qquad
        e^{-\ii h}=\frac{\sqrt B-
        \ii\sqrt A}{\sqrt{A+B}}.
\]
Therefore,
\[
        e^{\ii(\delta-h)}
        =\frac{a}{|a|}e^{-\ii h}
        =-
        \frac{c_p\sqrt B+
        \ii d_q\sqrt A}{\sqrt{(1-AB)(A+B)}}.
\]
Taking real and imaginary parts proves the claim.
\end{proof}

\begin{lemma}[Derivative form of the master inequality]\label{lem:derivative-master}
Let $U$ be an admissible zero of $G_{A,B}$. Then
\[
        W^2|K|\leq \frac{\pi^2}{2}
\]
is equivalent to
\[
        \frac1\pi G'_{A,B}(U)\geq \sqrt{2(1+AB)}.
\]
\end{lemma}

\begin{proof}
If $AB=1$, then $A=B=1$, $\kappa=\eps=0$, the admissible zero is $U=1/2$, and the assertion reduces to
\[
        2\sin(\pi/2)=\sqrt{2(1+1)}.
\]
Thus, we assume $AB<1$.

Since
\[
        \tan^2 h=\frac AB,
        \qquad
        \cos h=\sqrt{\frac B{A+B}},
        \qquad
        \sin h=\sqrt{\frac A{A+B}},
\]
Lemma~\ref{lem:phase-a} and the exact phase formula yield
\[
\begin{aligned}
&-\sqrt{1-AB}\bigl[
        \cos h\sin(\pi V)\cos(\delta-h)
        -\sin h\sin(\pi T)\sin(\delta-h)
        \bigr]             \\
&\hspace{4cm}=
        \frac{B\kappa\sin(\pi M)+A\eps\sin(\pi N)}{A+B}.
\end{aligned}
\]
Substitution into the master inequality gives
\[
        \sin(\pi U)
        +\frac{B\kappa\sin(\pi M)+A\eps\sin(\pi N)}{A+B}
        \geq
        \frac{\sqrt{2(1+AB)}}{A+B}.
\]
Multiplication by $A+B$ gives
\[
\begin{aligned}
 &(A+B)\sin(\pi U)
        +B\kappa\sin(\pi M)
        +A\eps\sin(\pi N)      \\
 &\hspace{4cm}\geq \sqrt{2(1+AB)}.
\end{aligned}
\]
By the derivative formula in Lemma~\ref{lem:monotone}, this is precisely
\[
        \frac1\pi G'_{A,B}(U)\geq \sqrt{2(1+AB)}.
\]
\end{proof}

\section{Proof of the scalar derivative inequality}\label{sec:scalar-proof}

This section proves Theorem~\ref{thm:scalar}.  The proof is elementary but rather rigid.  The idea is to replace the sine inequality by two linear lower bounds and then to obtain those lower bounds from a barrier point for the monotone function $G_{A,B}$.

Throughout this section we assume
\[
        0<A,B\leq 1,
        \qquad
        B\geq B_0(A),
\]
and we write
\[
        \kappa=\sqrt{1-A^2},
        \qquad
        \eps=\sqrt{1-B^2},
        \qquad
        P=\frac{1+AB}{B(A+B)}.
\]
Thus
\[
        M(U)=\kappa(P-U),
\]
\[
        N(U)=\eps\left(U+\frac{\kappa^2}{A(A+B)}\right),
\]
and
\[
        G_{A,B}(U)=B\cos(\pi M(U))-A\cos(\pi N(U))-(A+B)\cos(\pi U).
\]

Let $U$ be the admissible zero of $G_{A,B}$ and set
\[
\begin{aligned}
        S&=(A+B)\sin(\pi U)
        +B\kappa\sin(\pi M(U))
        +A\eps\sin(\pi N(U)).
\end{aligned}
\]
By Lemma~\ref{lem:domain},
\[
        0\leq M(U),N(U)\leq \frac12.
\]
Using
\[
        \sin(\pi x)\geq 2x,
        \qquad 0\leq x\leq \frac12,
\]
and
\[
        \sin(\pi U)\geq 2\min\{U,1-U\},
        \qquad 0\leq U\leq 1,
\]
we obtain
\[
        S\geq 2\left[(A+B)\min\{U,1-U\}+B\kappa M(U)+A\eps N(U)\right].
\]
Therefore, it is enough to prove the two linear estimates
\begin{equation}\label{eq:HR-linear}
        (A+B)(1-U)+B\kappa M(U)+A\eps N(U)
        \geq \frac12\sqrt{2(1+AB)},
\end{equation}
\begin{equation}\label{eq:HL-linear}
        (A+B)U+B\kappa M(U)+A\eps N(U)
        \geq \frac12\sqrt{2(1+AB)}.
\end{equation}

\subsection{The right linear estimate}

Put
\[
        x=P-U.
\]
Then
\[
        M(U)=\kappa x.
\]
Furthermore,
\[
        P+\frac{\kappa^2}{A(A+B)}
        =\frac{1+AB}{B(A+B)}+
        \frac{1-A^2}{A(A+B)}
        =\frac1{AB}.
\]
Hence
\[
        N(U)=\eps\left(\frac1{AB}-x\right).
\]
Define
\[
        H_R=(A+B)(1-U)+B\kappa M(U)+A\eps N(U).
\]
Substituting $U=P-x$, $M=\kappa x$, and $N=\eps(1/(AB)-x)$ gives the exact simplification
\[
        H_R=B(2+AB-A^2)x.
\]
Let
\[
        C=2+AB-A^2,
        \qquad
        \sigma=\sqrt{2(1+AB)}.
\]
Then \eqref{eq:HR-linear} is equivalent to
\[
        BCx\geq \frac\sigma2.
\]
Set
\[
        x_*:=\frac{\sigma}{2BC},
        \qquad
        U_*:=P-x_*.
\]
It remains to prove $U\leq U_*$. If $U_*\geq R(A,B)$, this is immediate. Assume therefore that
\[
        U_*<R(A,B).
\]
We shall prove
\[
        G_{A,B}(U_*)\geq 0.
\]
Since $G_{A,B}$ is strictly increasing on the admissible interval and $G_{A,B}(U)=0$, this implies $U\leq U_*$.

First, we check that $U_*$ lies in the admissible interval. We prove
\[
        U_*\geq \frac12.
\]
This is equivalent to
\[
        x_*\leq P-\frac12.
\]
But
\[
        P-\frac12=\frac{2+AB-B^2}{2B(A+B)}.
\]
Thus, it suffices to prove
\[
        \sigma(A+B)\leq C(2+AB-B^2).
\]
Since $\sigma\leq 2$, it is enough to know
\[
        2(A+B)\leq C(2+AB-B^2),
\]
that is,
\[
        Y(A,B):=(2+AB-A^2)(2+AB-B^2)-2(A+B)\geq 0.
\]
This is Lemma~\ref{lem:bernstein-Y}. Hence $U_*\geq 1/2$. By Lemma~\ref{lem:domain}, $L(A,B)\leq 1/2$, so $U_*\geq L(A,B)$.  Together with the standing assumption $U_*<R(A,B)$, this shows that $U_*$ is admissible.

Put
\[
        y=1-U_*.
\]
Because $U_*\geq 1/2$ and $U_*<R\leq 1$, we have
\[
        0\leq y\leq \frac12.
\]
At $U=U_*$,
\[
        M(U_*)=\kappa x_*,
        \qquad
        U_*=1-y,
\]
so
\[
        -\cos(\pi U_*)=\cos(\pi y).
\]
Therefore
\[
        G_{A,B}(U_*)
        =B\cos(\pi\kappa x_*)-A\cos(\pi N(U_*))+(A+B)\cos(\pi y).
\]
Using $\cos(\pi N(U_*))\leq 1$, one has
\[
        G_{A,B}(U_*)
        \geq
        B\cos(\pi\kappa x_*)-A+(A+B)\cos(\pi y).
\]
For $0\leq y\leq 1/2$,
\[
        \cos(\pi y)\geq 1-2y,
\]
and, for all real $t$,
\[
        \cos t\geq 1-\frac{t^2}{2}.
\]
Hence
\[
        \cos(\pi\kappa x_*)\geq 1-\frac{\pi^2}{2}\kappa^2 x_*^2.
\]
Thus
\[
        G_{A,B}(U_*)
        \geq
        B\left(1-\frac{\pi^2}{2}\kappa^2 x_*^2\right)
        -A+(A+B)(1-2y).
\]
Now
\[
        y=1-U_*=1-P+x_*.
\]
Since
\[
        1-P=-\frac{1-B^2}{B(A+B)}=-\frac{\eps^2}{B(A+B)},
\]
we get
\[
        y=x_*-\frac{\eps^2}{B(A+B)}.
\]
Substitution yields
\[
        G_{A,B}(U_*)
        \geq
        \frac2B-2(A+B)x_*-\frac{\pi^2}{2}B\kappa^2 x_*^2.
\]
Using $x_*=\sigma/(2BC)$ and $\sigma^2=2(1+AB)$, the right-hand side is nonnegative provided
\[
        2C^2-C(A+B)\sigma-
        \frac{\pi^2}{4}(1-A^2)(1+AB)\geq 0.
\]
Since $\pi^2<10$, it is enough to prove
\[
        2C^2-C(A+B)\sigma-
        \frac52(1-A^2)(1+AB)\geq 0.
\]
We use
\[
        \sigma\leq 2,
        \qquad
        2-\sigma=
        \frac{4-\sigma^2}{2+
        \sigma}
        =\frac{2(1-AB)}{2+
        \sigma}
        \geq \frac{1-AB}{2}.
\]
Therefore
\[
\begin{aligned}
        2C-(A+B)\sigma
        &=2(C-A-B)+(A+B)(2-
        \sigma)        \\
        &\geq 2(C-A-B)+\frac{A+B}{2}(1-AB).
\end{aligned}
\]
It remains to use Lemma~\ref{lem:bernstein-Z}, which asserts the nonnegativity of
\[
        Z(A,B)
        :=C\left[2(C-A-B)+\frac{A+B}{2}(1-AB)\right]
        -\frac52(1-A^2)(1+AB).
\]
Thus $G_{A,B}(U_*)\geq 0$. Since $G_{A,B}$ is increasing, we have $U\leq U_*$. Equivalently, $x=P-U\geq x_*$. Hence
\[
        H_R=BCx\geq BCx_*=
        \frac\sigma2.
\]
This proves \eqref{eq:HR-linear}.

\subsection{The left linear estimate}

The second estimate follows by symmetry. Under
\[
        A\leftrightarrow B,
        \qquad
        U\leftrightarrow 1-U,
\]
one has
\[
        M_{B,A}(1-U)=N_{A,B}(U),
        \qquad
        N_{B,A}(1-U)=M_{A,B}(U),
\]
and
\[
        G_{B,A}(1-U)=-G_{A,B}(U).
\]
Thus, if $U$ is the admissible zero for $(A,B)$, then $1-U$ is the admissible zero for $(B,A)$. By Lemma~\ref{lem:domain}, the swapped pair also satisfies the domain condition. Applying \eqref{eq:HR-linear} to $(B,A)$ gives
\[
        (A+B)U+A\eps N(U)+B\kappa M(U)
        \geq \frac12\sqrt{2(1+AB)},
\]
which is exactly \eqref{eq:HL-linear}.

\subsection{Completion of Theorem~\ref{thm:scalar}}

The two estimates \eqref{eq:HR-linear} and \eqref{eq:HL-linear} imply
\[
        (A+B)\min\{U,1-U\}+B\kappa M(U)+A\eps N(U)
        \geq \frac\sigma2.
\]
Combining this with the sine lower bound gives
\[
        S\geq 2\cdot \frac\sigma2=\sigma=
        \sqrt{2(1+AB)}.
\]
This is the desired scalar derivative inequality, and Theorem~\ref{thm:scalar} is proved.

\section{Completion of the Hopf estimate}\label{sec:completion}

\begin{proof}[Proof of Theorem~$\ref{thm:scherk-main}$]
Let $(p,q)$ be an admissible pair in the Scherk-type comparison family, and let $\zc$ be the distinguished zero of the harmonic projection. The harmonic-measure construction produces an admissible zero $U$ of $G_{A,B}$ with
\[
        A=\sin p,
        \qquad
        B=\sin(q-p).
\]
By Lemma~\ref{lem:domain}, the corresponding admissible interval is nonempty, so $B\geq B_0(A)$. The scalar theorem gives
\[
        \frac1\pi G'_{A,B}(U)\geq \sqrt{2(1+AB)}.
\]
By Lemma~\ref{lem:derivative-master}, this is equivalent to the Hopf master inequality, and hence 
\[
        W^2|K|\leq \frac{\pi^2}{2}
\]
for the Scherk-type comparison graph.
\end{proof}

\begin{proof}[Proof of Theorem~$\ref{thm:main}$]
By scaling, it suffices to consider the unit disk. Translate vertically so that the point above the center is $(0,0,0)$, and rotate the horizontal coordinates so that the mixed second derivative at the center is zero. Choose the normal parameter $w$ corresponding to the unit normal at the center.

The Scherk comparison theorem gives a comparison graph $S^\diamond$ with the same normal at the distinguished point and with
\[
        W_S(\xi)^2 |K_S(\xi)|
        <
        W_{S^\diamond}(\xi)^2 |K_{S^\diamond}(\xi)|.
\]
By Theorem~\ref{thm:scherk-main},
\[
        W_{S^\diamond}(\xi)^2 |K_{S^\diamond}(\xi)|
        \leq \frac{\pi^2}{2}.
\]
Thus
\[
        W_S(\xi)^2 |K_S(\xi)|
        <\frac{\pi^2}{2}
\]
for the unit disk. If the original disk has radius $R$, the homothety to the unit disk leaves $W$ unchanged and multiplies $|K|$ by $R^2$. Therefore
\[
        W_S(\xi)^2 |K_S(\xi)|
        < \frac{\pi^2}{2R^2}.
\]

The constant is sharp because it is already sharp in the horizontal tangent-plane class of Finn and Osserman \cite{FinnOsserman1964}.  In that class $W=1$ at the center, so the normalized estimate reduces to their sharp curvature estimate.  Hence the constant $\pi^2/2$ cannot be lowered.
\end{proof}

\subsection{Figures}

The following figures illustrate the admissible parameter domain and the corresponding numerical behavior of the normalized curvature in the Scherk-type family. The estimate itself is proved above; the figures are included only to visualize the parameter region and the shape of the curvature function.

\begin{figure}[p]
\centering

\IfFileExists{Plot2D.pdf}{%
\includegraphics[width=0.62\textwidth,height=0.36\textheight,keepaspectratio]{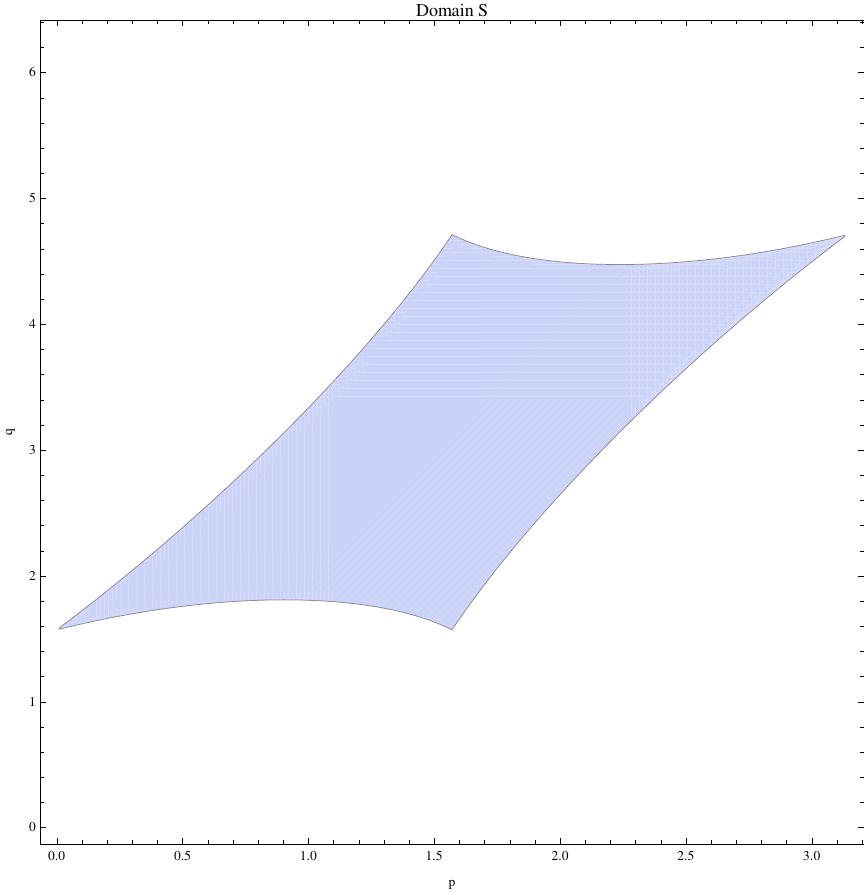}%
}{%
\fbox{\begin{minipage}[c][2.2in][c]{0.75\textwidth}\centering Plot2D.pdf not found.\end{minipage}}%
}
\caption{The admissible parameter domain $\mathcal S$ in the $(p,q)$-plane.}
\label{fig:domain}

\vspace{0.6cm}

\IfFileExists{plot3D.pdf}{%
\includegraphics[width=0.72\textwidth,height=0.42\textheight,keepaspectratio]{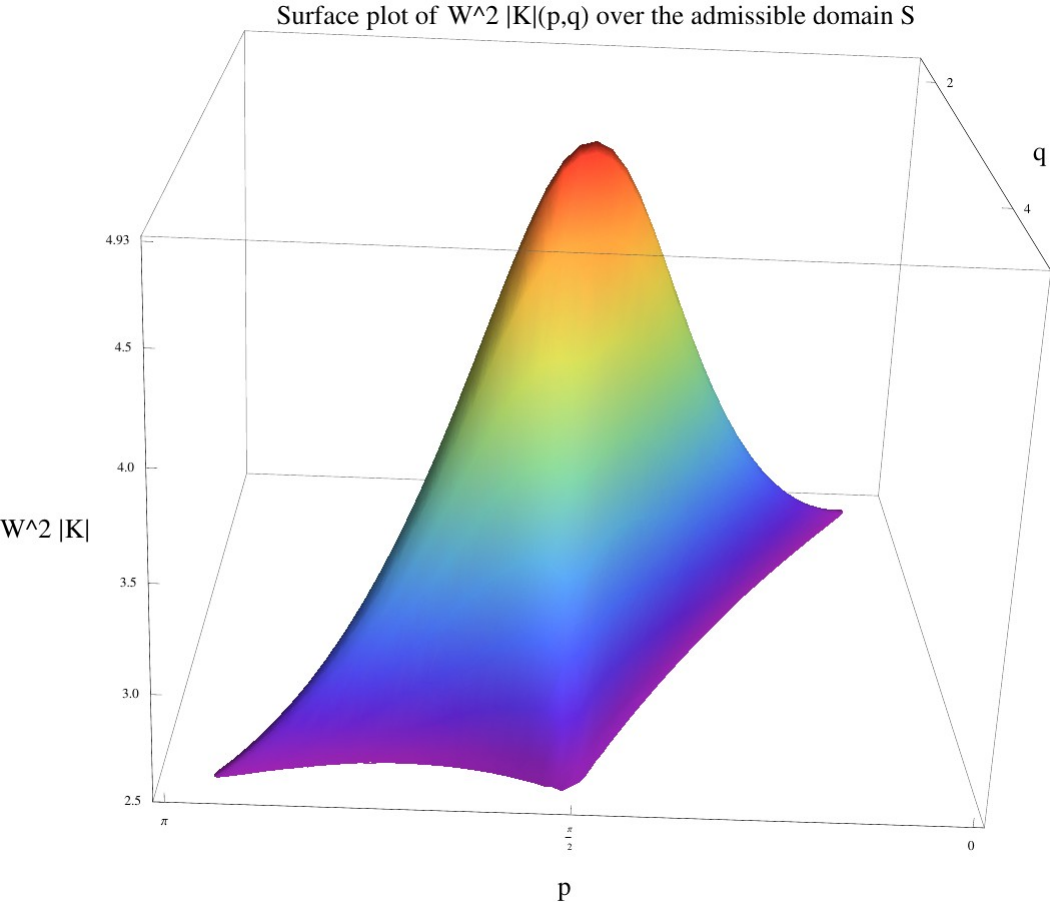}%
}{%
\fbox{\begin{minipage}[c][2.2in][c]{0.75\textwidth}\centering plot3D.pdf not found.\end{minipage}}%
}
\caption{A numerical plot of $W^2|K|(p,q)$ over the admissible domain. It illustrates the proved two-sided estimate
$\pi^2/4\leq W^2|K|\leq \pi^2/2$ in the Scherk-type comparison family.}
\label{fig:surface}

\end{figure}

\clearpage
\appendix

\section{An auxiliary odd coefficient estimate}\label{app:odd}

This appendix records a sharp coefficient estimate that gives a quick proof of the Hopf estimate in a centrally symmetric boundary-regular setting. It is not used in the proof of the main theorem.

\begin{lemma}[Odd Hall averaging inequality]\label{lem:odd-Hall}
Let $\theta:\mathbb R\to\mathbb R$ be nondecreasing and satisfy
\[
        \theta(x+\pi)=\theta(x)+\pi .
\]
For $0\leq \tau\leq \pi/2$, define
\[
        J(\tau)
        =
        \frac1{2\pi}\int_0^{2\pi}
        \sin^2\frac{\theta(x+\tau)-\theta(x-\tau)}2\,\dd x .
\]
If $M$ is nonnegative and nonincreasing on $[0,\pi/2]$, then
\[
        \int_0^{\pi/2} M(\tau)J(\tau)\,\dd\tau
        \leq
        \frac2\pi
        \int_0^{\pi/2} M(\tau)\tau\,\dd\tau .
\]
\end{lemma}

\begin{proof}
This is the odd-period analogue of Hall's Lemmas~2 and~3
in~\cite{Hall1982}. We give the details, since the form needed here is
not stated explicitly there.

First observe that if $\alpha_1,\ldots,\alpha_N\geq 0$ and
\[
        \alpha_1+\cdots+\alpha_N\leq \pi,
\]
then
\[
        \sum_{n=1}^N \sin^2\frac{\alpha_n}{2}\leq 1.
\]
Indeed, whenever $\alpha,\beta\geq 0$ and $\alpha+\beta\leq\pi$,
\[
        \sin^2\frac{\alpha}{2}
        +
        \sin^2\frac{\beta}{2}
        \leq
        \sin^2\frac{\alpha+\beta}{2},
\]
and the assertion follows by repeatedly combining terms.

Now let
\[
        \tau_1+\cdots+\tau_N=\frac{\pi}{2}.
\]
Choose functions $y_n=y_n(x)$ such that
\[
        y_1=x,
        \qquad
        y_n-\tau_n=y_{n-1}+\tau_{n-1},
        \qquad 2\leq n\leq N.
\]
Set
\[
        \alpha_n(x)
        =
        \theta(y_n+\tau_n)-\theta(y_n-\tau_n).
\]
Then
\[
\begin{aligned}
        \sum_{n=1}^N\alpha_n(x)
        &=
        \theta(y_N+\tau_N)-\theta(y_1-\tau_1)  \\
        &=\pi,
\end{aligned}
\]
because
\[
        (y_N+\tau_N)-(y_1-\tau_1)
        =
        2(\tau_1+\cdots+\tau_N)
        =
        \pi
\]
and $\theta(u+\pi)=\theta(u)+\pi$. Hence
\[
        \sum_{n=1}^N
        \sin^2\frac{\alpha_n(x)}2
        \leq 1.
\]
For each fixed $n$, the function
\[
        x\mapsto
        \sin^2\frac{\theta(x+\tau_n)-\theta(x-\tau_n)}2
\]
is $\pi$-periodic. Indeed, by $\theta(x+\pi)=\theta(x)+\pi$,
\[
        \theta(x+\pi+\tau_n)-\theta(x+\pi-\tau_n)
        =
        \theta(x+\tau_n)-\theta(x-\tau_n).
\]
Since $y_n=x+\mathrm{const}$, integration over $0\leq x\leq 2\pi$
therefore gives
\[
        \frac1{2\pi}\int_0^{2\pi}
        \sin^2\frac{\theta(y_n+\tau_n)-\theta(y_n-\tau_n)}2\,\dd x
        =
        J(\tau_n).
\]
Hence, after integrating
\[
        \sum_{n=1}^N
        \sin^2\frac{\alpha_n(x)}2
        \leq 1
\]
over $0\leq x\leq 2\pi$, we obtain
\[
        J(\tau_1)+\cdots+J(\tau_N)\leq 1.
\]

Put
\[
        A=\frac{\pi}{2},
        \qquad
        L(\tau)=J(\tau)-\frac{\tau}{A}
        =
        J(\tau)-\frac{2\tau}{\pi},
\]
and define
\[
        I(M)=\int_0^A M(\tau)L(\tau)\,\dd\tau .
\]
It is enough to prove $I(M)\leq 0$.

For $B=A/2^r$ and $0\leq \tau\leq B/2$, apply the preceding estimate to
$2^r$ copies of $\tau$ and $2^r$ copies of $B-\tau$. Since their total
sum is $A$, we get
\[
        2^r\{J(\tau)+J(B-\tau)\}\leq 1.
\]
Equivalently,
\[
        L(\tau)+L(B-\tau)\leq 0.
\]

We now use Hall's dyadic folding argument. Let $M_0=M$. Having defined
a nonnegative nonincreasing function $M_r$ supported on $[0,A/2^r]$, set
\[
        M_{r+1}(\tau)
        =
        \begin{cases}
        M_r(\tau)-M_r(A/2^r-\tau),
        &0\leq \tau\leq A/2^{r+1},\\[4pt]
        0,
        &A/2^{r+1}<\tau\leq A.
        \end{cases}
\]
Then $M_{r+1}$ is again nonnegative and nonincreasing, and the preceding
two-point inequality gives
\[
        I(M_r)\leq I(M_{r+1})
        \qquad r=0,1,2,\ldots .
\]
Hence
\[
        I(M)\leq I(M_r)
        \qquad r=0,1,2,\ldots .
\]

It remains to let $r\to\infty$. Since
\[
        \sin^2\frac{u}{2}\leq \frac{u}{2},
        \qquad u\geq 0,
\]
we have
\[
        J(\tau)
        \leq
        \frac12\cdot
        \frac1{2\pi}\int_0^{2\pi}
        \bigl(\theta(x+\tau)-\theta(x-\tau)\bigr)\,\dd x
        =
        \tau .
\]
Here we used
\[
\int_0^{2\pi}\bigl(\theta(x+\tau)-\theta(x-\tau)\bigr)\,dx=4\pi\tau,
\]
which follows from \(\theta(x+2\pi)=\theta(x)+2\pi\).
Hence, by changes of variables,
\[
\begin{aligned}
\int_0^{2\pi}\!\bigl(\theta(x+\tau)-\theta(x-\tau)\bigr)\,\dd x
&=
\int_\tau^{2\pi+\tau}\theta(u)\,\dd u
-
\int_{-\tau}^{2\pi-\tau}\theta(u)\,\dd u        \\
&=
\int_{2\pi-\tau}^{2\pi+\tau}\theta(u)\,\dd u
-
\int_{-\tau}^{\tau}\theta(u)\,\dd u             \\
&=
\int_{-\tau}^{\tau}\theta(v+2\pi)\,\dd v
-
\int_{-\tau}^{\tau}\theta(v)\,\dd v             \\
&=
\int_{-\tau}^{\tau}\bigl(\theta(v)+2\pi\bigr)\,\dd v
-
\int_{-\tau}^{\tau}\theta(v)\,\dd v             \\
&=4\pi\tau .
\end{aligned}
\]
Therefore
\[
        \frac12\cdot
        \frac1{2\pi}\int_0^{2\pi}
        \bigl(\theta(x+\tau)-\theta(x-\tau)\bigr)\,\dd x
        =
        \frac12\cdot\frac{4\pi\tau}{2\pi}
        =
        \tau .
\]
Thus
\[
        I(M_r)
        \leq
        \int_0^{A/2^r} M_r(\tau)J(\tau)\,\dd\tau
        \leq
        \int_0^{A/2^r} M(\tau)\tau\,\dd\tau .
\]
If the right-hand side in the desired inequality is infinite, there is
nothing to prove. Otherwise the last integral tends to $0$ as
$r\to\infty$. Therefore $I(M)\leq 0$, that is,
\[
        \int_0^{\pi/2} M(\tau)J(\tau)\,\dd\tau
        \leq
        \frac2\pi
        \int_0^{\pi/2} M(\tau)\tau\,\dd\tau .
\]
\end{proof}

\begin{theorem}[Odd boundary coefficient estimate]\label{thm:odd-coeff}
Let
\[
        f=h+\overline g
\]
be a sense-preserving harmonic diffeomorphism of $\D$ onto $\D$, with
\[
        h(z)=\sum_{n=0}^\infty a_nz^n,
        \qquad
        g(z)=\sum_{n=1}^\infty b_nz^n .
\]
Assume that $f$ extends to an orientation-preserving homeomorphism of
$\overline\D$ and that its boundary values are odd:
\[
        f(-\zeta)=-f(\zeta),
        \qquad |\zeta|=1 .
\]
Then
\[
        |a_1|^2+|b_1|^2\geq \frac{8}{\pi^2}.
\]
The constant is sharp.
\end{theorem}

\begin{proof}
Write the boundary values as
\[
        F(e^{\ii t})=f(e^{\ii t})=e^{\ii\theta(t)},
\]
where $\theta$ is a nondecreasing lift satisfying
\[
        \theta(t+2\pi)=\theta(t)+2\pi .
\]
Since the boundary map is odd, the lift may be chosen so that
\[
        \theta(t+\pi)=\theta(t)+\pi .
\]

Let
\[
        c_n=\frac1{2\pi}\int_0^{2\pi}
        F(e^{\ii t})e^{-\ii nt}\,\dd t
\]
be the Fourier coefficients of $F$. Since $F$ is odd, all even Fourier
coefficients vanish. Moreover, from the Poisson representation of
$f=h+\overline g$,
\[
        c_1=a_1,
        \qquad
        c_{-1}=\overline{b_1}.
\]
For $n\geq 1$, set
\[
        S_n=|c_n|^2+|c_{-n}|^2 .
\]

Hall's autocorrelation identity gives
\[
        C(t):=\frac1{2\pi}\int_0^{2\pi}
        \cos\bigl(\theta(s+t)-\theta(s-t)\bigr)\,\dd s
        =
        \sum_{n=1}^{\infty} S_n\cos(2nt).
\]
Since only odd Fourier modes occur,
\[
        C(t)=\sum_{k=0}^{\infty}
        S_{2k+1}\cos\bigl(2(2k+1)t\bigr).
\]
Therefore, by orthogonality on $[0,\pi/2]$,
\[
        S_1=
        \frac4\pi\int_0^{\pi/2} C(t)\cos(2t)\,\dd t .
\]
The odd-mode expansion also gives
\[
        C(\pi/2-t)=-C(t).
\]
Hence
\[
        S_1=
        \frac8\pi\int_0^{\pi/4} C(t)\cos(2t)\,\dd t .
\]

Define
\[
        J(t)=
        \frac1{2\pi}\int_0^{2\pi}
        \sin^2\frac{\theta(s+t)-\theta(s-t)}2\,\dd s .
\]
Then
\[
        C(t)=1-2J(t).
\]
Consequently,
\[
        S_1
        =
        \frac8\pi\int_0^{\pi/4}
        \bigl(1-2J(t)\bigr)\cos(2t)\,\dd t .
\]

Apply Lemma~\ref{lem:odd-Hall} with
\[
        M(t)=
        \begin{cases}
        \cos(2t), & 0\leq t\leq \pi/4,\\
        0, & \pi/4<t\leq \pi/2.
        \end{cases}
\]
This $M$ is nonnegative and nonincreasing on $[0,\pi/2]$. Hence
\[
        \int_0^{\pi/4} J(t)\cos(2t)\,\dd t
        \leq
        \frac2\pi\int_0^{\pi/4} t\cos(2t)\,\dd t .
\]
Therefore
\[
\begin{aligned}
        S_1
        &\geq
        \frac8\pi\int_0^{\pi/4}\cos(2t)\,\dd t
        -
        \frac{16}{\pi}\cdot\frac2\pi
        \int_0^{\pi/4}t\cos(2t)\,\dd t  \\
        &=
        \frac8{\pi^2}.
\end{aligned}
\]
Because
\[
        S_1=|c_1|^2+|c_{-1}|^2
        =|a_1|^2+|b_1|^2,
\]
the desired estimate follows.

It remains to show sharpness. Consider the four-point boundary collapse
\[
        F_\square(e^{\ii t})=\ii^k,
        \qquad
        \frac{k\pi}{2}\leq t<\frac{(k+1)\pi}{2},
        \qquad k=0,1,2,3.
\]
This map is not itself a boundary homeomorphism, but it is the
$L^1(\partial\D)$-limit of odd orientation-preserving circle
homeomorphisms. By the Radó--Kneser--Choquet theorem, the Poisson
extensions of these approximating homeomorphisms are sense-preserving
harmonic diffeomorphisms of $\D$ onto $\D$.

The Poisson extension of $F_\square$ has, up to rotation,
\[
        h'(z)=\frac{2(1-\ii)}{\pi}\frac1{1-z^4},
        \qquad
        g'(z)=-\frac{2(1+\ii)}{\pi}\frac{z^2}{1-z^4}.
\]
Thus
\[
        a_1=h'(0)=\frac{2(1-\ii)}{\pi},
        \qquad
        b_1=g'(0)=0.
\]
Consequently,
\[
        |a_1|^2+|b_1|^2
        =
        \left|\frac{2(1-\ii)}{\pi}\right|^2
        =
        \frac8{\pi^2}.
\]
The approximating odd boundary homeomorphisms therefore show that the
constant $8/\pi^2$ cannot be improved.
\end{proof}

\begin{theorem}[Centrally symmetric graphs, boundary-regular form]\label{thm:central}
Let
\[
        S=\{(u,v,U(u,v)):(u,v)\in\D\}
\]
be a non-parametric minimal graph, and let $\xi=(0,0,U(0,0))$. Assume
\[
        U(-u,-v)=2U(0,0)-U(u,v).
\]
Assume in addition that the horizontal harmonic projection in a conformal parametrization extends to an odd orientation-preserving boundary homeomorphism of $\overline\D$. Then
\[
        W(\xi)^2|K(\xi)|\leq \frac{\pi^2}{2}.
\]
\end{theorem}

\begin{proof}
Choose conformal parameters $X:\D\to S$, $X(0)=\xi$, and write
\[
        X(z)=(\operatorname{Re} f(z),\operatorname{Im} f(z),T(z)),
        \qquad
        f=h+\overline g.
\]
By the assumed odd boundary extension and Theorem~\ref{thm:odd-coeff},
\[
        |a_1|^2+|b_1|^2\geq \frac8{\pi^2}.
\]
For a minimal graph, the Enneper--Weierstrass data can be written as
\[
        P=h',
        \qquad
        \omega=\frac{g'}{h'}=q^2,
\]
where $q:\D\to\D$ is analytic. Thus
\[
        g'=q^2P.
\]
The curvature and slope factor are
\[
        K=-\frac{4|q'|^2}{|P|^2(1+|q|^2)^4},
        \qquad
        W=\frac{1+|q|^2}{1-|q|^2}.
\]
Put $r=|q(0)|$. Then
\[
        W(0)^2|K(0)|
        =\frac{4|q'(0)|^2}{|P(0)|^2(1-r^2)^2(1+r^2)^2}.
\]
By Schwarz--Pick,
\[
        |q'(0)|\leq 1-r^2.
\]
Therefore
\[
        W(0)^2|K(0)|
        \leq \frac{4}{|P(0)|^2(1+r^2)^2}.
\]
Now
\[
        a_1=P(0),
        \qquad
        b_1=q(0)^2P(0),
\]
so
\[
        |a_1|^2+|b_1|^2=|P(0)|^2(1+r^4)
        \leq |P(0)|^2(1+r^2)^2.
\]
Using the coefficient estimate,
\[
        |P(0)|^2(1+r^2)^2\geq \frac8{\pi^2}.
\]
Substituting gives
\[
        W(0)^2|K(0)|\leq \frac4{8/\pi^2}=\frac{\pi^2}{2}.
\]
\end{proof}

\section{The lower bound in the Scherk family}\label{app:lower}

The derivative form gives a convenient expression for the normalized curvature. Let
\[
\begin{aligned}
        S_{A,B}(U)
        &=(A+B)\sin(\pi U)
        +B\kappa\sin(\pi M(U))
        +A\eps\sin(\pi N(U)).
\end{aligned}
\]
Then the Scherk-family identity can be written as
\[
        W^2|K|=\frac{\pi^2(1+AB)}{S_{A,B}(U)^2}.
\]
The scalar theorem says $S_{A,B}(U)\geq \sqrt{2(1+AB)}$, which gives the upper bound. The following estimate gives the lower bound.

\begin{proposition}[Lower bound]\label{prop:lower}
For every admissible Scherk-type comparison graph,
\[
        W^2|K|\geq \frac{\pi^2}{4}.
\]
Consequently,
\[
        \frac{\pi^2}{4}\leq W^2|K|\leq \frac{\pi^2}{2}
\]
throughout the admissible Scherk-type family.
\end{proposition}

\begin{proof}
By admissibility,
\[
        0\leq U\leq 1,
        \qquad
        0\leq M(U)\leq U,
        \qquad
        0\leq N(U)\leq 1-U.
\]
Thus all sine factors in $S_{A,B}(U)$ are at most $1$, and
\[
        S_{A,B}(U)\leq A+B+B\kappa+A\eps.
\]
We claim that
\[
        A+B+B\kappa+A\eps\leq 2\sqrt{1+AB}.
\]
Indeed, using $A^2+
        \kappa^2=1$ and $B^2+
        \eps^2=1$, one verifies the identity
\[
\begin{aligned}
&4(1+AB)-(A+B+B\kappa+A\eps)^2       \\
&\hspace{2cm}=(\kappa\eps+
        \kappa+
        \eps-1-AB)^2\geq 0.
\end{aligned}
\]
Therefore
\[
        S_{A,B}(U)^2\leq 4(1+AB).
\]
Since
\[
        W^2|K|=\frac{\pi^2(1+AB)}{S_{A,B}(U)^2},
\]
we obtain
\[
        W^2|K|\geq \frac{\pi^2}{4}.
\]
The upper bound is Theorem~\ref{thm:scherk-main}.
\end{proof}

\section{Bernstein-polynomial checks}\label{app:bernstein}

We use the following elementary positivity criterion.

\begin{lemma}[Bernstein positivity]\label{lem:bernstein}
Let $P(t,v)$ be a polynomial on $[0,1]^2$. Suppose that for some bidegree $(m,n)$ it has a Bernstein expansion
\[
        P(t,v)=\sum_{i=0}^m\sum_{j=0}^n
        p_{ij}\binom mi t^i(1-t)^{m-i}\binom nj v^j(1-v)^{n-j}
\]
with all $p_{ij}\geq 0$. Then $P(t,v)\geq 0$ on $[0,1]^2$.
\end{lemma}

\begin{lemma}\label{lem:bernstein-Y}
For $0\leq A,B\leq 1$,
\[
        Y(A,B):=(2+AB-A^2)(2+AB-B^2)-2(A+B)\geq 0.
\]
\end{lemma}

\begin{proof}
Put
\[
        A=1-t,
        \qquad
        B=1-v,
        \qquad
        0\leq t,v\leq 1.
\]
Then $Y(1-t,1-v)$ has the Bernstein expansion of bidegree $(3,3)$,
\[
        Y(1-t,1-v)=\sum_{i=0}^3\sum_{j=0}^3
        y_{ij}\binom3i t^i(1-t)^{3-i}\binom3j v^j(1-v)^{3-j},
\]
where
\[
        (y_{ij})=
        \begin{pmatrix}
        0 & 2/3 & 1/3 & 0\\
        2/3 & 2 & 20/9 & 2\\
        1/3 & 20/9 & 28/9 & 10/3\\
        0 & 2 & 10/3 & 4
        \end{pmatrix}.
\]
All coefficients are nonnegative, so $Y\geq 0$ by Lemma~\ref{lem:bernstein}.
\end{proof}

\begin{lemma}\label{lem:bernstein-Z}
Let
\[
        C=2+AB-A^2
\]
and define
\[
        Z(A,B)=C\left[2(C-A-B)+\frac{A+B}{2}(1-AB)\right]
        -\frac52(1-A^2)(1+AB).
\]
Then
\[
        Z(A,B)\geq 0,
        \qquad 0\leq A,B\leq 1.
\]
\end{lemma}

\begin{proof}
Put
\[
        A=1-t,
        \qquad
        B=1-v,
        \qquad
        0\leq t,v\leq 1.
\]
Then $2Z(1-t,1-v)$ has the Bernstein expansion of bidegree $(4,4)$,
\[
        2Z(1-t,1-v)=\sum_{i=0}^4\sum_{j=0}^4
        z_{ij}\binom4i t^i(1-t)^{4-i}\binom4j v^j(1-v)^{4-j},
\]
where
\[
        (z_{ij})=
        \begin{pmatrix}
        0 & 1 & 4/3 & 5/4 & 1\\
        0 & 9/8 & 41/24 & 15/8 & 7/4\\
        10/3 & 55/12 & 49/9 & 143/24 & 37/6\\
        5 & 103/16 & 23/3 & 139/16 & 19/2\\
        5 & 13/2 & 8 & 19/2 & 11
        \end{pmatrix}.
\]
All coefficients are nonnegative, so $Z\geq 0$ by Lemma~\ref{lem:bernstein}.
\end{proof}

\begin{remark}
The displayed Bernstein expansions are polynomial identities. They can be verified directly by expanding both sides in the monomial basis.
\end{remark}

\section{Log-subharmonicity of the normalized curvature}\label{app:log}

Let
\[
        X(z)=\operatorname{Re}\int^z
        \bigl(1-g^2,\ii(1+g^2),2g\bigr)\Phi,
        \qquad
        \Phi=\varphi(z)\,\dd z,
\]
be a conformal minimal immersion. Then
\[
        \dd s^2=(1+|g|^2)^2|\varphi|^2|\dd z|^2,
\]
and
\[
        K=-\frac{4|g'|^2}{|\varphi|^2(1+|g|^2)^4}.
\]
For an upward oriented graph, $|g|<1$, and
\[
        W=\frac{1+|g|^2}{1-|g|^2}.
\]
Consequently,
\[
        W^2|K|=
        \frac{4|g'|^2}{|\varphi|^2(1-|g|^4)^2}.
\]

\begin{lemma}[Log-subharmonicity]\label{lem:log-subharmonicity}
On the set where $g'\neq 0$,
\[
        \Delta\log(W^2|K|)
        =\frac{32|g|^2|g'|^2}{(1-|g|^4)^2}\geq 0.
\]
Thus, $W^2|K|$ is log-subharmonic away from branch points of $g$; across zeros of $g'$ the statement holds in the distributional sense.
\end{lemma}

\begin{proof}
We have
\[
        \log(W^2|K|)
        =\log4+2\log\left|\frac{g'}{\varphi}\right|
        -2\log(1-|g|^4).
\]
Away from zeros of $g'$, the middle term is harmonic. A direct radial calculation gives
\[
        \Delta\bigl[-2\log(1-|g|^4)\bigr]
        =\frac{32|g|^2|g'|^2}{(1-|g|^4)^2}.
\]
The extension across zeros of $g'$ follows from the distributional subharmonicity of $\log|g'|$.
\end{proof}

\begin{remark}
By contrast,
\[
        \Delta\log|K|
        =-
        \frac{16|g'|^2}{(1+|g|^2)^2}
        \leq 0
\]
away from zeros of $g'$. Thus, $|K|$ is log-superharmonic in the Gauss-map variable, whereas $W^2|K|$ is log-subharmonic. This observation is useful conceptually, but it does not replace the zero equation $f(\zc)=0$, because $\zc$ itself depends on the Scherk parameters.
\end{remark}

\end{document}